\documentclass[11pt,fleqn]{article}
\usepackage[utf8]{inputenc}
\usepackage{amsmath,amsfonts,amsthm,amssymb,epsfig,color}

\newcommand{\R}{\mathbb R}
\newcommand{\C}{\mathbb C}

\newcommand{\Z}{\mathbb Z}
\newcommand{\D}{\mathbb D}
\newcommand{\tr}{\mathop {\rm trace}\nolimits}
\newcommand{\diag}{\mathop {\rm diag}\nolimits}
\renewcommand{\Re}{\mathop {\rm Re}\nolimits}

\newtheorem{theorem}{Theorem}
\newtheorem{corollary}[theorem]{Corollary}
\newtheorem{lemma}[theorem]{Lemma}

\title{Matrices with invariant by rotation numerical ranges}

\author{Michel Crouzeix\footnote{Univ.\,Rennes, CNRS, IRMAR\,-\,UMR\,6625, F-35000 Rennes, France.
email: michel.crouzeix@univ-rennes.fr},}

\begin{document}

\maketitle
\begin{abstract}
We characterize the $d\times d$ matrices whose numerical ranges are invariant by rotations of angle $2\pi/d$.
\end{abstract}

\section{Introduction}
Here we consider a square matrix $A\in\C^{d,d}$ whose the numerical range $W(A)$ is invariant by a rotation of angle $2\pi /d$. Without loss of generality, it suffices to look at the case of rotation centered at the origin $0$. We want to characterize such a matrix. For that, we introduce the condition
\[
(C_d)\qquad\left\{\begin{matrix}
\textrm{There exists a polynomial  } P\in\R[w]\ \textrm{such that} \hfill\\[3pt]
\det\big(w\,I-\tfrac12(e^{-i\theta}A{+}e^{i\theta}A^*))\big)=P(w)-\big(\tfrac{-1}2\big)^{d-1}\Re\big(e^{-id\theta }\det(A)\big).\hfill
\end{matrix}
\right.
\]
A first  characterization will be given by the theorem

\begin{theorem}\label{th1}
 Assume that $A\in\C^{d,d}$, $d\geq 3$. Then, for having the numerical range $W(A)$ invariant by the rotation centered at $0$ of angle $2\pi /d$, the condition $(C_d)$ is sufficient. If $\det(A)\neq 0$, this condition is necessary and then we have $A^d=(-1)^{d-1}\det(A)I$ and furthermore $W(A)$ has the same symmetry group as a regular $d$-gon.
\end{theorem}
We will see that, if $\det(A)=0$, the condition $(C_d)$ implies that $W(A)$ is a disk. Conversely, 
if $W(A)$ is a disk centered at $0$, then it is known that $\det(A)=0$; but, for all $d\geq 3$, there exist matrices not satisfying $(C_d)$ such that $W(A)$ is a disk. \medskip

\noindent{\bf Remark.}{\it We do not consider the trivial case $d=2$. All $2\times2$ matrices have their numerical ranges invariant by the rotation of angle $\pi $ with the same center as the ellipse $W(A)$.}
\medskip
 
The condition $(C_d)$ is not simple to verify, we have found it easier to use the following
 \begin{theorem}\label{th2}
 Assume that $A\in\C^{d,d}$, $d\geq 3$. Then, the condition $(C_d)$ is equivalent to
\begin{align*}
\tr\big((e^{-i\theta }A{+}e^{i\theta }A^*)^k\big)&=\textrm{constant,  for }k=1,\dots,d{-}1,\\
\textrm{and\quad} \tr\big((e^{-i\theta }A{+}e^{i\theta }A^*)^d\big)&=e^{-id\theta }\tr(A^d){+}e^{id\theta }\tr(A^{*d}){+}\textrm{constant}.
\end{align*}
\end{theorem}

The two next sections are devoted to the proofs of these theorems. Section\,4 is concerned with a description of the boundary $\partial W(A)$ of the numerical range when the condition $(C_d)$ is satisfied and when we are not in the trivial cases where $W(A)$ is a disk or a polygon. Then the boundary is $C^1$ and without loss of generality we may assume that $(-1)^{d-1}\det(A)>0$.
We look at three different descriptions of the boundary $\partial W(A)$ 
\begin{align*}
\partial W(A)&=\{z\,; z=\zeta (\theta )\,,\theta \in\R\},\quad \partial W(A)=\{z\,; z=r(\psi  )e^{i\psi },\psi  \in\R\},\\
\partial W(A)&=\{z\,; z=\sigma (e^{it})\,,t \in\R\}.
\end{align*}
The first comes from the tangential representation, $\theta $ is a.e. the angle of the exterior unit normal $\nu(\theta )=\cos\theta {+}i\sin\theta $ to the boundary; the second corresponds to the polar representation, $r(\cdot)$ is a positive-valued function;
for the third, $\sigma $ is the Riemann mapping from the closed unit disk $\overline\D$ onto $W(A)$ satisfying $\sigma (0)=0$ and $\sigma '(0)>0$. In each case, the functions $\theta \mapsto e^{-i\theta } \zeta (\theta )$, $\psi \mapsto r(\psi )$ and $t\mapsto e^{-it}\sigma (e^{it})$ are $C^1$, even and $\frac{2\pi }d$-periodic. The maps $\theta \mapsto \arg(\zeta (\theta))$ and $t\mapsto \arg(\sigma (e^{it}))$ are increasing on $\R$; the maps
 $\theta \mapsto |\zeta (\theta)|$, $\psi \mapsto r(\psi )$, and $t\mapsto |\sigma (e^{it})|$ are decreasing
 on the interval $[0,\frac\pi d]$. When the boundary has no straight parts, these functions are single-valued and analytic on $\R$. 
Otherwise, there exists $\psi _0\in(0,\frac\pi d)$ such that the arc of boundary $\{z\in \partial W(A)\,; -\psi _0\leq \arg(z)\leq \psi _0\}$ has an analytic representation by $\zeta $, $r$ or $\sigma $ and the part 
$\{z\in \partial W(A)\,; \psi _0\leq \arg (z)\leq \frac{2\pi }d{-} \psi _0\}$ is a straight segment orthogonal to $\nu(\frac\pi d)$. 

In Section\,5, we consider a family of matrices obtained by multiplying a permutation matrix by a diagonal matrix ; for them the invariance by rotation of the numerical range is well-known \cite{choi,iss,tsaiwu}.
We show that this is the unique family satisfying $(C_d)$ if $d=3$ (up to a unitary similarity). In Section\,6, we will see that many other matrices have this invariance property. 
We provide the algebraic equations which characterize them for $d=4$: we get four cubic equations with nine real parameters.
We do not think it is possible to obtain explicit formulae for a general solution, but we describe some sub-families of solutions.

\section {Proof of Theorem\,\ref{th1}}

Before going to the proof, we recall some facts concerning the tangential representation of the numerical range. We define
\begin{align}\label{defT}
T(u,v,w)=\det(u\,B{+}v\,C{+}w\,I),\quad\textrm{where }B=\frac12(A{+}A^*),\ C=\frac1{2i}(A{-}A^*).
\end{align}
The polynomial $T(u,v,w)$ is homogeneous of degree $d$, with real coefficients in the variables $u,v,w$, the coefficient of $w^d$ being $1$. For each  $(u,v)\in \R^2$, the equation $T(u,v,w)=0$ has 
$d$ real roots that are the eigenvalues of the self-adjoint matrix $-(uB{+vC)}$.
Note also that $T(-1,-i,\lambda)=\det(\lambda I{-}A)$, and thus the eigenvalues of the matrix $A$ are
the roots of the equation $T(-1,-i,\lambda)=0$.
\medskip

We denote by $w_M(u,v)$ the greatest root of  $T(u,v,w)$ and let $w_M(\theta )=w_M(-\cos \theta ,-\sin\theta )$. Thus, $w_M(\theta )$ is the greatest root of $T(-\cos\theta ,-\sin\theta ,w)=
\det\big(w\,I-\tfrac12(e^{-i\theta}A{+}e^{i\theta}A^*)\big)=0$, i.e, $w_M(\theta )$ is the greatest 
eigenvalue of the self-adjoint matrix $M(\theta )=\frac12(e^{-i\theta }A{+}e^{i\theta }A^*)$; thus for all unit vectors $\alpha\in\C^d$, it holds $\langle M(\theta )\alpha,\alpha\rangle\leq w_M(\theta )$
\footnote{ $\langle x,y\rangle$ denotes the usual inner product in $\C$} and there exists at least one unit vector $\alpha$ which realizes  $\langle M(\theta )\alpha,\alpha\rangle=w_M(\theta )$. 
Setting $z=x{+}iy=\langle A\alpha,\alpha\rangle$ and noticing that  $\langle M(\theta )\alpha,\alpha\rangle=\Re (e^{-i\theta}z)=x\,\cos\theta{+}y\,\sin\theta$, it can be seen that all $z\in W(A)$ belong to the half plane $\Pi _\theta :=\{x{+}iy\,; x\,\cos\theta{+}y\,\sin\theta\leq w_M(\theta )\}$, and that there exists at least
one contact point between $W(A)$ and the straight line $\Delta _\theta := \{x{+}iy\,; x\,\cos\theta{+}y\,\sin\theta=w_M(\theta )\}$ boundary of $\Pi_\theta $. We set $\zeta (\theta )=W(A)\cap\Delta _\theta$, and thus either $\zeta (\theta )$ is reduced to the unique contact point between $W(A)$ and the tangent $\Delta _\theta $, or $\zeta (\theta )$ is a segment $[\zeta _-(\theta ),\zeta _+(\theta )]$ with $\arg(\zeta _-(\theta ))<\arg(\zeta _+(\theta ))$.
We have
 \begin{align*}
W(A)&=\{z\in\C\,; z\in \Pi_\theta ,\ \forall \theta \in[0,2\pi ]\}\quad\textrm{and}\quad \partial W(A)=\{\zeta (\theta )\,; \theta \in[0,2\pi ]\}.
\end{align*}
From this, it is clear that $W(A)$ is invariant by the rotation centered at $0$ of angle $2\pi /d$ if and only if $w_M$ is periodic with period $2\pi /d$, and that $W(A)$ is symmetric with respect to the straight line
$e^{i\varphi }\R$ if and only if $w_M(\varphi {-}\theta )=w_M(\varphi {+}\theta )$ for all $\theta $.
\begin{proof}[Proof of Theorem\,\ref{th1}]
If $A$ satisfies the condition $(C_d)$
\[
\det\big(w\,I-\tfrac12(e^{-i\theta}A{+}e^{i\theta}A^*))\big)=P(w)-\big(\tfrac{-1}2\big)^{d-1}\Re\big(e^{-id\theta }\det(A)\big),
\]
then it is clear that $w_M$ is periodic with period $2\pi /d$, thus $W(A)$ is invariant by the rotation centered at $0$ of angle $2\pi /d$ and it is symmetric with respect to the line $e^{i\varphi }\R$ with $\varphi =\arg(\det(A))$. \smallskip

For the necessary part, we assume this rotation invariance for $W(A)$ and $\det(A)\neq 0$, (this implies that $W(A)$ is not a disk since otherwise $0$ would be an eigenvalue of $A$).  Then, we set $A_k=e^{2ik\pi /d}A$,
 $B_k=\frac12(A_k{+}A_k^*)$, $C_k=\frac1{2i}(A_k{-}A_k^*)$, \\$T_k(u,v,w)=\det(uB_k{+}vC_k+wI)$ and $S(u,v,w)=\sum_{k=0}^{d-1}T_k(u,v,w)$. Note that 
 \[
 T_k(-\cos\theta,-\sin\theta ,w)=T(-\cos(\theta{+}\tfrac{2k\pi }{d}) ,-\sin(\theta{+}\tfrac{2k\pi }{d}) ,w).
 \]
 We write
 \begin{align*}
 T(-\cos\theta ,-\sin\theta ,w)&=w^d{+}t_1(\theta )w^{d-1}{+}\cdots{+}t_d(\theta ),\\
S(-\cos\theta ,-\sin\theta ,w)&=\frac{1}{d}\sum_{k=0}^{d-1}T(-\cos(\theta{+}\tfrac{2k\pi }{d}) ,-\sin(\theta{+}\tfrac{2k\pi }{d})  ,w)\\
&=w^d{+}s_1(\theta )w^{d-1}{+}\cdots{+}s_d(\theta ).
 \end{align*}
Recall that $T(-\cos\theta ,-\sin\theta ,w)=\det(w\,I{-}\tfrac12(e^{-i\theta }A{+}e^{i\theta }A^*))$, therefore we can write
\[
t_1(\theta )=t_{10}e^{-i\theta }{+}t_{11}e^{i\theta} , \quad 
t_2(\theta )=t_{20}e^{-2i\theta }{+}t_{21}{+}t_{22}e^{2i\theta },\quad t_k(\theta )=\sum_{j=0}^k t_{kj}\,e^{-(k-2j)i\theta}.
\]
Clearly, $S(-\cos\theta ,-\sin\theta ,w)$, $s_1(\theta ),\dots,s_d(\theta )$ are invariant by the translation $\theta \mapsto \theta {+}\frac{2\pi }d$. Therefore they are functions of $e^{id\theta }$ whence, if $j<d$, we have $s_j(\theta )=\sigma _j$ (constant);  if $j=d$, $s_d(\theta )=s_{d0}\,e^{-id\theta }{+}\overline {s_{d0}}\,e^{id\theta }{+}\sigma _d$. It is easily seen that $s_{d0}=t_{d0}$, and
from $t_d(\theta )=\det(\tfrac12(e^{-i\theta }A{+}e^{i\theta }A^*))$, we deduce $t_{d0}=\frac1{2^d}\det(A)$. 

From the invariance by rotation of $W(A)$, the largest root of $T(-\cos\theta ,-\sin\theta ,\cdot)$
is a root of $T_k(-\cos\theta ,-\sin\theta ,\cdot)$ thus a root of $S(-\cos\theta ,-\sin\theta ,\cdot)$.
This shows that, if $u^2{+}v^2=1$, then $w_M(u,v)$ ($=w_M(\theta )$ for some $\theta $) is a common root for 
$T(u,v,w)$ and $S(u,v,w)$; from the homogeneity, we deduce that $w_M(u,v)$ is a common root for all $u,v$. Now, we turn to show that $S(u,v,\cdot)=T(u,v,\cdot)$. These two polynomials share the root $w_M$, hence, considering them as polynomials with coefficients in the field of rational functions in $u,v$, they have a g.c.d., $Q(u,v,w)$, which is a unitary polynomial of degree $r$ in $w$, $1\leq r\leq d$. Since the ring $\R[u,v]$ of polynomials is factorial, we deduce that $Q(u,v,w)$ is a homogeneous polynomial of degree $r$ in $u,v,w$.\medskip

Remark that $w_M(\theta)$ is a continuous function in $\theta$, periodic of period $2\pi/d$\,: value $w_{min}$ is obtained  at $d$ different angles and the maximum $w_{max}$ for $d$ other values. Therefore, for each intermediate $x$ with
$w_{min}\leq x<w_{max}$, there exist $2d$ different values of $\theta\in [0,2\pi)$ such that 
$w_M(\theta)=x$. In other terms the system
\begin{align*}
Q(u,v,x)=0,\qquad u^2+v^2=1,
\end{align*}
has $2d$ different solutions and a finite total number of solutions, which requires that the polynomial 
$Q(u,v,x)$ be of degree at least $d$ with respect to $u,v$, and therefore $T=Q=S$.
We write $T=S$ in homogeneous form 
\[
T(u,v,w)=w^d+\sum_{k\leq d/2}a_{2k}(u^2{+}v^2)^kw^{d-2k}+\Big(\big(\tfrac{u{-}iv}2\big)^d\det(A)+\big(\tfrac{u{+}iv}2\big)^d\det(A^*)\big),
\]
and note that $\det(wI{-}A)=T(-1,-i,w)=w^d-(-1)^{d-1}\det A$, then we deduce the property
$A^d=(-1)^{d-1}\det(A)I$ from the Cayley-Hamilton Theorem.\medskip
\end{proof}

\noindent{\bf Remark.} {\it If $\det(A)=0$ and $d\geq 3$, we still have $A^d=0$, but there exist matrices $d\times d$ invariant by the rotation centered at $0$ of angle $2\pi /d$ which do not satisfy Condition $(C_d)$. For instance with $A=\begin{pmatrix}0 &2 &0\\0 &0 &0\\0 &0 &a\end{pmatrix}$, $a\in(0,1]$, $W(A)$ is the closed unit disk and we have
$\det(\tfrac12(e^{-i\theta }A{+}e^{i\theta }A^*){+}w\,I))=w^3+a\cos\theta\ w^2-w-a\cos\theta $.}

\section {Proof of Theorem\,\ref{th2}}

We introduce the notations
\begin{align*}
B(\theta )=\frac 12(e^{-i\theta}A{+}e^{i\theta}A^*),\quad q_k(\theta )=\tr\big(B(\theta )^k\big).
\end{align*}
Recall that $\det(B(\theta ){+}wI)=w^d{+}t_1(\theta )w^{d-1}{+}\cdots{+}t_d(\theta )$.
Thus, if $\mu_1(\theta ),\dots,\mu_d(\theta )$ denote the eigenvalues of $B(\theta )$, then
\begin{align*}
t_k(\theta )&=\sum_{i_1<\cdots<i_k}\mu_{i_1}(\theta )\cdots\mu_{i_k}(\theta ),\quad
q_k(\theta )=\sum_{j=1}^d(\mu_j(\theta ))^k.
\end{align*}
Thus, we have the relations
\begin{align*}
&q_1(\theta)-t_1(\theta)=0,\\
&q_k(\theta)+\sum_{j=1}^{k-1}(-1)^jt_j(\theta) \,q_{k-j}(\theta)+(-1)^k k\,t_k(\theta)=0,\quad\hbox{if  }1\leq k\leq d.
\end{align*}
It is easily seen by recursion from these relations that ``$t_k(\theta )$ constant for $k=1,\dots,d{-}1$" is equivalent to ``$q_k(\theta )$ constant for $k=1,\dots,d{-}1$". When this is realized, then
``$t_d(\theta )$ only depends on $e^{\pm id\theta}$" is clearly equivalent to ``$q_d(\theta)$ only depends on $e^{\pm id\theta}$". Furthermore, it is clear  that the term in $e^{-id\theta }$ in 
$\tr\big((e^{-i\theta }A{+}e^{i\theta }A^*)^d\big)$ is $e^{-id\theta }\tr(A^d)$. Altogether, this proves Theorem\,2.\medskip

\noindent{\bf Remark.} {\it We have also the relations
\[
q_k(\theta)+\sum_{j=1}^{k-1}(-1)^jt_j(\theta) \,q_{k-j}(\theta)=0,\quad\hbox{ for }k>d.
\]
Therefore we deduce that Condition $(C_d)$ implies for all $k>d$ : the trigonometric polynomial 
$\tr\big((e^{-i\theta}A{+}e^{i\theta}A^*)^k\big)$ has only terms in $e^{ijd\theta}$ with $j\in\Z$.
}\medskip

Note that the condition $A^d=(-1)^{d-1}\det(A)\,I$ implies $\tr A^k=0$, for $1\leq k\leq d-1$.
Therefore we deduce easily that, if $A\in\C^{d,d}$ and $\det(A)\neq 0$, then

$\bullet$ If $d=3$, then $A$ is invariant by the rotation centered at $0$ of angle $2\pi /3$ if and only if
$A^3=\det(A)\,I$ and $\tr(A^*A^2)=0$,

$\bullet$ If $d=4$, then $A$ is invariant by the rotation centered at $0$ of angle $\pi /2$ if and only if
$A^4=-\det(A)\,I$, $\tr(A^*A^2)=0$, and $\tr(A^*A^3)=0$.

$\bullet$ If $d=5$, then $A$ is invariant by the rotation centered at $0$ of angle $2\pi /5$ if and only if
$A^5=-\det(A)\,I$, $\tr(A^*A^2)=0$, $\tr(A^*A^3)=0$, $\tr(A^*A^4)=0$ and  $\tr(A^{*2}A^3)+\tr(A^*AA^*A^2)=0$.

\section{About the boundary of the numerical range}
  
  In this section, we assume that $\det(A)\neq 0$ and that Condition $(C_d)$ is satisfied; thus the numerical range is not a disk. Since we  have $W(\lambda \,A)=\lambda W(A)$,  there will be no loss of generality to assume $\det(A)=(-1)^{d-1}$, thus $A^d=I$. Recall that $w_M(\theta )$ is the largest root of $T(-\cos\theta ,-\sin \theta ,w)$, i.e. the largest solution of $P(w)=\frac{1}{2^{d-1}}\cos(d\theta )$, and the largest eigenvalue of $\frac12(e^{-i\theta}A{+}e^{i\theta }A^*)$, whence the function $w_M(\theta )$ is continuous, and periodic of period $2\pi /d$; this function is also even since $W(A)$ is symmetric with respect to the real axis. We have seen that $\partial W(A)=\{\zeta (\theta )\,; \theta \in\R\}$ where $\zeta (\theta )=W(A)\cap\Delta _\theta $, thus $\zeta (\theta {+}\frac{2\pi }d)=e^{2i\pi /d}\zeta (\theta )$ and $\zeta (-\theta )=\overline{\zeta (\theta )}$. Therefore, for describing the global properties of $\zeta (\theta )$, it often will be sufficient to consider its behavior on the interval $[-\frac\pi d,\frac\pi d]$ or on the interval $[0,\frac{2\pi }d]$. 
  
 Now, we introduce the generalized Crawford number $w_*=\min\{|z|\,; z\in\partial W(A)\}$. Clearly, this minimum is realized by some $z_*\in\partial W(A)$ and there exists $\theta_*$ such that $w_*=w_M(\theta_*)=\min\{w_M(\theta )\,; \theta \in\R\}$. Note that $w_*$ is the largest of the $d$ roots of $P(w)=\frac{1}{2^{d-1}}\cos(d\theta^*)$ which are real, thus the $d{-}1$ roots of $P'$ are real and less than or equal to $w_*$. Since $P$ is unitary, we deduce
 $P'(w)>0$ for all $w>w_*$ and $P'(w_*)\geq 0$.
 From $w_*\leq w_M(\theta )$, we get $P(w_*)\leq P(w_M(\theta ))=\frac{1}{2^{d-1}}\cos(d\theta )$ for all $\theta $, therefore the minimum $w_*=\min\{w_M(\theta )\,; \theta \in\R\}$ is realized if and only if $\theta =\frac\pi d$ mod\big($\frac{2\pi }d\big)$.
  
If $\partial W(A)$ has a corner, thus $d$ corners, then these corners are eigenvalues of $A$ (see for instance \cite[Theorem\,13]{kip2}) whence they are the $d$-roots of the unit since $A^d=I$, and 
$A$ is unitarily similar to $\diag(1,e^{2i\pi /d},\dots,e^{2(d-1)i\pi /d})$. Then, $W(A)$ is the polygon with these eigenvalues as vertices. From now on, we will assume that $\partial W(A)$ has no corner.

From the implicit function theorem applied to the equation $P(w)=\frac{1}{2^{d-1}}\cos(d\theta )$, if we have $P'(w_M(\theta ))\neq 0$ (thus $P'(w_M(\theta ))>0$) then the function $w_M(\theta )$ is analytic in $\theta $. From $P'(w_M(\theta ))w'_M(\theta )=-\frac{d}{2^{d-1}}\sin{d\theta }$, we obtain $w_M'(\theta )>0$, thus $w_M$ is increasing from $w_*$ to $w_M(0)$ in the interval $(-\frac\pi d,0)$ and it is decreasing ($w_M'(\theta )<0$) from $w_M(0)$ to $w_*$ on $(0,\frac\pi d)$. Now, if $z=x{+}iy\in \zeta (\theta )= \Delta _\theta \cap W(A)$, then we have $x\,\cos\theta {+}y\,\sin\theta =w_M(\theta)$ and $x\,\cos(\theta {+}\varepsilon){+}y\,\sin(\theta{+}\varepsilon ) \leq w_M(\theta{+}\varepsilon )$, for all $\varepsilon $ since $z\in W(A)\subset\Pi _{\theta +\varepsilon }$. From $\cos(\theta {+}\varepsilon )-\cos \theta =-\varepsilon \sin\theta {+}\textrm{O}(\varepsilon ^2)$ and $\sin(\theta {+}\varepsilon )-\sin \theta =\varepsilon \cos\theta {+}\textrm{O}(\varepsilon ^2)$,
 we get $-\varepsilon\,x\,\sin\theta +\varepsilon y\cos\theta \leq \varepsilon w_M'(\theta ){+}
  \textrm{O}(\varepsilon ^2)$. Taking the limit as $\varepsilon \to 0$ with $\varepsilon >0$ and with $\varepsilon <0$ we deduce $-x\,\sin\theta + y\cos\theta=w_M'(\theta )$. Therefore we obtain that $z=x{+}iy=(w_M(\theta ){+}i\,w_M'(\theta ))e^{i\theta }=\zeta (\theta )$ is the unique point of $\Delta _\theta \cap W(A)$. In particular we have $\zeta (0)=w_M(0)$ and, if 
$P'(w_*)\neq 0$, $\zeta(\frac\pi d)=w_*e^{i\pi /d}$. Furthermore (with $ B(0,r):=\{z\in\C\,; |z|\leq r\}$)
\[
B(0,w_*)\subset W(A)\subset B( 0,w^*),\quad \textrm{where } w^*=w_M(0)=\max\{w_M(\theta )\}.
\] 

Now, on the interval $(-\frac\pi d,\frac\pi d)$, we use the branch $a(\theta )$ of $\arg(\zeta (\theta ))$ which is continuous and satisfies $a(0)=0$. From $\zeta (\theta )=(w_M(\theta ){+}i\,w_M'(\theta ))e^{i\theta }$, we have $\theta -\frac\pi 2<a(\theta )<\theta $ if $0<\theta <\frac\pi d$ since
$w'_M(\theta )<0$; on the other side, since $w^*$ belongs to the half-plane $\Pi _\theta $, it holds
\[
w^*\cos\theta \leq w_M(\theta )=|\delta (\theta)|\cos(\theta {-}a(\theta ))<w^*\cos(\theta {-}a(\theta )),
\]
which shows that $0<a(\theta )<\theta $. Similarly, we have $\theta <a(\theta )<0$ if $-\frac\pi d<\theta <0$.
Then, we remark that the map $\theta \mapsto a(\theta)$ is one to one; indeed, since $W(A)$ is convex and $0$ is interior to $W(A)$, then $a(\theta_1)=a(\theta_2)$ implies $\zeta (\theta _1)=\zeta (\theta _2)$ which induces $\theta _1=\theta _2$ since $\zeta (\theta _1)$ is not a corner. This shows that $a(\theta )$ is monotone and thus increasing with $\theta$. A simple calculation shows that $a'(\theta )=\frac{w_M(\theta )(w_M(\theta )+w_M''(\theta ))}{w_M(\theta )^2+w_M''(\theta )^2}$, whence we obtain
$w_M(\theta )+w_M''(\theta)\geq 0$ on $(-\frac\pi d,\frac\pi d)$. Now we consider 
$\frac d{d\theta }|\delta (\theta )|^2=2\,w'_M(\theta )(w_M(\theta )+w_M''(\theta )).$
This shows that $|\delta (\theta )|$ is increasing on $(-\frac\pi d,0)$ and decreasing on $(0,\frac\pi d)$. 

We have the following lemma  (recall that $\partial W(A)=\{\zeta (\theta )\,; \theta \in\R\}$).

 \begin{lemma}  \label{lem3} 
 If $P'(w_*)\neq 0$. Then the function $w_M(\theta )$ is analytic, even, and periodic of period $2\pi /d$. The boundary of the numerical range is described by $\zeta(\theta )=(w_M(\theta ){+}i\,w_M'(\theta ))e^{i\theta }$. The argument of $\zeta (\theta )$ has a determination which is $C^\infty$ and increasing in $\theta $. The $C^\infty$ map $\theta \mapsto |\zeta (\theta )|$ increases from $w_*$ to $w^*$ on $(-\frac\pi d,0)$ and decreases from $w^*$ to $w_*$ on $(0,\frac\pi d)$.
 
If $P'(w_*)= 0$ and if $W(A)$ has no corner. Then the function $w_M(\theta )$ is continuous, even, and periodic of period $2\pi /d$, its restriction to the segment $[-\frac\pi d,\frac\pi d]$ is analytic. The boundary of the numerical range is described by :  if $\theta \neq \pm\frac\pi d$, then $\zeta(\theta )$ is the point $(w_M(\theta ){+}i\,w_M'(\theta ))e^{i\theta }$ and, if $\theta = \frac\pi d$, then $\zeta(\frac\pi d )$ is the line segment $\{(w_*{+}itw_*')e^{i\pi /d }\,; -1\leq t\leq 1\}$, where $w_*'$ is the left limit in $\frac\pi d$ of $w_M'(\theta)$. The argument of $\zeta (\theta )$ has a determination which is $C^\infty$ and increasing on the interval $(-\frac\pi d,\frac\pi d)$. The $C^\infty$ map $\theta \mapsto |\zeta (\theta )|$ increases from $\sqrt{w_*^2+w_*'^2}$ to $w^*$ on $(-\frac\pi d,0)$ and decreases from $w^*$ to  $\sqrt{w_*^2+w_*'^2}$ on $(0,\frac\pi d)$.

If $W(A)$ has corners. Then $w_M(\theta)=\cos\theta $, $\zeta (\theta )=1$ is constant on the interval $(-\frac\pi d,\frac\pi d)$ and $\zeta (\frac\pi d)$ is the line segment $[1,\exp(\frac{2i\pi}d)]$; furthermore $P'(w_*)=0$.
\end{lemma}
\begin{proof} It remains to look at the case $P'(w_*)= 0$; we first assume that $W(A)$ has no corner.
Since $P'(w_M)\neq 0$, $w_M$ is analytic on the interval $(-\frac\pi d,\frac\pi d)$ and $\delta (\theta )=
(w_M(\theta ){+}i\,w_M'(\theta ))e^{i\theta }$ describes a part of the boundary $\partial W(A)$ which is contained in the sector $\{z\in\C\,; |\arg(z)|\leq \frac\pi d\}$. As before, $\arg(\zeta (\theta))$ increases on $(-\frac\pi d,\frac\pi d)$, $|\zeta (\theta )|$ increases on $(-\frac\pi d,0)$ and decreases on $(0,\frac\pi d)$.

 We now consider the situation close to the value $\frac\pi d$; for that
we set $\theta =\frac\pi d{+}\varepsilon $. The solutions of $P(w(\theta ))=\frac{1}{2^{d-1}}\cos(d\theta )$
close to $w_*=w_M(\frac\pi d)$ satisfy
\[
P(w(\theta )){-}P(w_*)=2^{2-d}\sin^2\tfrac{d\varepsilon }2=\tfrac12\,P''(w_*)(w(\theta ){-}w_*)^2+\textrm{O}((w(\theta ){-}w_*)^3).
 \] 
We have $P''(w_*)>0$, indeed otherwise this equation would have non real solutions and (Puiseux theory) there exist two (and only two) solutions $ w_+$ and $w_-$ in a neighbor of $\frac\pi d$, these solutions are real for real $\theta$ and holomorphic. They satisfy 
 \begin{align*}
w_\pm(\theta )&=w_*\pm \gamma_1\tfrac{d\theta {-}\pi }2+ \textrm{O}((\theta {-}\tfrac\pi d)^2),\quad \textrm{with  }\gamma_1=  (2^{d-3}P''(w_*)\big)^{-1/2}, \\[3pt]
w'_\pm(\theta )&=\pm\tfrac d2\gamma_1+ \textrm{O}((\theta {-}\tfrac\pi d)). 
\end{align*} 
Therefore we have $w_M(\theta )=w_{-}(\theta )$, if $\theta\in(-\frac\pi d,\frac\pi d]$ and $w_M(\theta )=w_{+}(\theta )$, if $\theta \in[\frac\pi d,\frac{3\pi} d)$.
This shows that $w_M$ is analytic up to the boundary of the segment $[-\frac\pi d,\frac\pi d]$ and that
$\delta (\theta )=(w_M(\theta ){+}i\,w_M'(\theta ))e^{i\theta }$ (for $\theta \neq \frac\pi d$) has left and right limits $\delta _{\pm}(\frac\pi d)=(w_*\pm i\,w_*')e^{i\pi /d}$ when $\theta \to \frac\pi d$.
Clearly, $\zeta (\frac\pi d)$ is the line segment $\{(w_*{+}itw_*')e^{i\pi /d }\,; -1\leq t\leq 1\}$.

If $W(A)$ has corners, then $A$ is unitarily similar to $\diag(1,e^{2i\pi /d},\dots,e^{2(d-1)i\pi /d})$, hence
$P(w)=\prod_{j=0}^{d-1}\big(w-\cos(\theta -\frac{2j\pi}d)\big)$ which shows that $w_M(\theta )=\cos\theta $ if $\theta \in(-\frac\pi d,\frac\pi d)$ and that $2^{d-1}P(\cos\theta )=\cos(d\theta )$ (i.e. $2^{d-1}P$ is the Chebyshev polynomial of degree $d$). Therefore, we easily obtain that $P'(w_*)=0$, $\zeta (\theta )=1$ is constant and $\zeta (\frac\pi d)$ is the line segment $[1,\exp(\frac{2i\pi}d)]$.

\end{proof}

\begin{corollary}
If $A\in\C^{d,d}$ satisfies the condition $(C_d)$, then $W(A)$ has a line segment on the boundary if and only if the largest root of $P(\cdot){-}P(w_*)$ is equal to the largest root of $P'$.
\end{corollary}
  
\noindent{\bf Remark.}{\it A particular case where line segments occur on the boundary is when $A$ is unitarily similar to a direct sum $A=A_1\oplus\cdots \oplus A_{d_2}$ with $A_j=\exp(\frac{2\pi (j-1)}d)A_1$, $j=1,\dots,d_2$, and $A_1\in\C^{d_1,d_1}$ invariant by rotation of angle $\frac{2\pi }{d_1}$, $\det(A_1)=(-1)^{d_1-1}$; thus $d=d_1d_2$. Then, $W(A)$ is the convex hull of $W(A_1)\cup\cdots\cup W(A_{d_2})$, and it is easily seen that $\partial W(A)$ has a straight part in each sector $\arg(z) \in(\frac{2k\pi}d,\frac{2(k+1)\pi}d)$, $k=0,\dots,d{-}1$. But line segments may also appear with irreducible matrices satisfying $(C_d)$, examples with $d=5$ will be given at the end of the next section.}\medskip

We turn now to other descriptions of the boundary $\partial W(A)$. First we look at the polar form
\[
\partial W(A)=\{z\,; z=r(\psi )\,e^{i\psi }, \ \psi \in\R\},
\] 
where $r$ is an even, $\frac{2\pi }d$ periodic, continuous, real-valued function. Then, it is easily seen
from the properties of $\zeta (\theta )$ that\smallskip

$\bullet$ if $P'(w_*)\neq 0$, then $r(\psi )$ is an analytic function of $\psi $ and
$r(\cdot)$ is a decreasing function on $(0,\frac\pi d)$,\smallskip

$\bullet$ and if $P'(w_*)=0$, then $r(\psi )$ is an analytic function of $\psi $
for all $\psi \in [-\psi _d,\psi _d]$ with $\psi _d=\frac\pi d{-}\arctan\frac{\gamma_1}{w_1}$. In the straight part $\psi \in [\psi _d,\frac{2\pi} d{-}\psi_d]$ we have $r(\psi )=w_1/\cos(\frac\pi d{-}\psi )$
and we have a $C^1$ linking between these two analytic functions left and right to the value $\psi_d$, since the corresponding point is not a corner. Furthermore, $r(\cdot)$ is a decreasing function on $(0,\frac\pi d)$.\medskip

Another natural parametrization of the boundary $\partial W(A)$ is given via the Riemann mapping $\sigma  $ from  the closed unit disk onto $W(A)$ such that $\sigma  (0)=0$ and $\sigma '(0)>0$,
\[
\partial W(A)=\{\sigma(e^{i\,t})\,; t \in\R\}. 
\]
Clearly, there is a one to one correspondence (modulo $2\pi $) between $\psi  $ and $t$ 
corresponding to the same boundary point $r(\psi )\,e^{i\psi }=\sigma (e^{it})$; in particular $\psi =0$ corresponds to $t=0$ and $\psi =\frac\pi d$ corresponds to $t=\frac\pi d$. This implies that the function $|\sigma (e^{it})|$ is decreasing on the interval $(0,\frac\pi d)$. In cases where $\psi _d$ is defined, we denote as $t_d$ the corresponding value such that $r(\psi _d)e^{i\psi _d}=\sigma (e^{it_d})$.
We deduce the following lemma.

\begin{lemma}\label{lem4} We assume that $A\in\C^{d,d}$ has a numerical range invariant by the rotation centered at $0$ of angle $2\pi /d$, that $(-1)^{d-1}\det(A)>0$ and that $W(A)$ has no corners.
Then, the functions $\psi \mapsto r(\psi )$ and $t\mapsto e^{-it}\sigma (e^{it})$ are even $\frac{2\pi }d$ periodic functions; the functions $\psi \mapsto |r(\psi )|$ and $t\mapsto|\sigma (e^{it})|$ are
 decreasing on the interval $[0,\frac\pi d]$. If the boundary $\partial W(A)$ has no straight part,
 these functions are analytic, otherwise there exists $\psi _d$ and $t_d\in(0,\frac\pi d)$ such that $r$ is analytic on $[-\psi _d,\psi _d]$ and on $[\psi _d,\frac{2\pi }d-\psi _d]$ with a $C^1$ junction in $\psi _d$, and 
 $\sigma (e^{i\cdot})$ is analytic on $[-t_d,t_d]$ and on $[t_d,\frac{2\pi }d-t_d]$ with a $C^1$ junction in $t_d$, the intervals $[\psi _d,\frac{2\pi }d-\psi _d]$ and $[t_d,\frac{2\pi }d-t_d]$ corresponding to the straight parts. 
 \end{lemma}

 \section{Weighted shift matrices}\label{sec5}
 In this section $A$ is a $3\times 3$ matrix and $T(u,v,w)$ is the corresponding polynomial defined in \eqref{defT}.
 \begin{lemma}\label{lem6}
 We assume that $d=3$ and $\det(A)>0$. Then,
the numerical range $W(A)$ is invariant by the rotation of angle $2\pi/3$ centered at 0  if and only if $A$ is unitarily equivalent to a matrix of the form 
\begin{equation*}
M(\alpha_1,\alpha_2,\alpha_3)=e^{i\theta }\,\begin{pmatrix} 0 &\alpha_1 &0\\0&0&\alpha_2\\
\alpha_3&0&0
\end{pmatrix} \quad\textrm{with}\quad \alpha_1>0,\alpha_2>0,\alpha_3>0, \theta \in\R.
\end{equation*}
\end{lemma}
\begin{proof}
If $A$ is unitarily equivalent to $M(\alpha_1,\alpha_2,\alpha_3)$, then a simple calculation shows that 
\begin{align*}
T(-\cos\theta ,-\sin\theta ,w)=w^3-\tfrac14(\alpha_1^{\,2}{+}\alpha_2^{\,2}{+}\alpha_3^{\,2})\,w+\frac{\det(A)}4 \,\cos(3\,\theta ),
\end{align*}
whence $W(A)$ is invariant by the rotation of angle $2\pi/3$ centered at $0$. Conversely, if we assume that $W(A)$ is invariant by this rotation, then $A^3=\det(A)\,I$ and
(see Section\,3) this is equivalent
to $\tr(A)=0$, $\tr(A^2)=0$, and $\tr(A^*A^2)=0$. Also, from the remark in Section\,3, the trigonometric polynomial $\tr\big((e^{-i\theta}A+e^{i\theta}A^*)^5\big)$ has no term in $e^{i\theta }$, whence
\begin{align*}
0=\tr((A^*A)^2A+A^{*2}A^3)=\tr((A^*A)^2A)+\tr(A^{*2})=\tr((A^*A)^2A).
\end{align*}

Without loss of generality, we suppose that the matrix $A^*A=\diag(\mu_1,\mu_2,\mu_3)$ is diagonal.
We begin with the case where the $\mu_j$ are all distincts. Then, it follows from
$\tr (A)=0$, $\tr (A^*A^2)=0$ and $\tr ((A^*A)^2A)=0$ that $a_{11}=a_{22}=a_{33}=0$. We write that
$A^*A$ is diagonal, this gives
\begin{align}\label{z}
\overline {a_{12}} a_{13}=0,\quad\overline{ a_{21} }a_{23}=0,\quad \overline{ a_{31}} a_{32}=0.
\end{align}
Note that no column of $A$ may be null since 0 is not eigenvalue of $A$. Therefore, there are two
possibilities:
\begin{align*}
a_{12}\neq0,\quad a_{13}=0,\quad a_{23}\neq0,\quad a_{21}=0,\quad a_{31}\neq0,\quad a_{32}=0
,\\
\hbox{or}\qquad 
a_{13}\neq0,\quad a_{12}=0,\quad a_{32}\neq0,\quad a_{31}=0,\quad a_{21}\neq0,\quad a_{23}=0.
\end{align*}
In the first occurence we have $A=M(a_{12},a_{23},a_{31})$ which is unitarily similar for some $\theta \in\R$ to 
$e^{i\theta }\,M(|a_{12}|,|a_{23}|,|a_{31}|)$, in the second occurence we have 
$A=M(a_{21},a_{32},a_{13})^T$, which is unitarily similar to $M(a_{32},a_{21},a_{13})$.\medskip

We now consider the case where $A^*A$ has a double eigenvalue. We can assume that 
$\mu_1\neq\mu_2=\mu_3$ and $a_{21}=0$, this last relation being realized by replacing, if needed, $A$ by $U^*AU$ with $U$ being a rotation matrix acting only on the two last components.
From $\tr (A)=0$ and $\tr (A^*A^2)=0$ we deduce $a_{11}=0$ and $a_{22}{+}a_{33}=0$.
Since the first column is not null, we have $a_{31}\neq0$. Writing that $0=(A^*A)_{31}=\overline{a_{33}}
a_{31}$, we obtain $a_{33}=0$, whence $a_{22}=0$. 
The relations \eqref{z} are still valid and we conclude as previously.\medskip

Finally, if $\mu_1{=}\mu_2{=}\mu_3$, then the matrix $A$ is normal, thus unitarily equivalent to 
$c\,\diag(1,\jmath,\jmath^2)$ with $c=\det(A)^{1/3}$, which is unitarily equivalent to $M(c,c,c)$.
\end{proof}

More generally, Daeshik Choi \cite{choi} has studied the matrices of the form $M(\alpha_1,\dots,\alpha_d)=P_d\,\diag(\alpha_1,\dots,\alpha_d)$ for any $d\geq 2$, and given a simple proof that these matrices have their numerical ranges invariant by  the rotation of angle $2\pi/d$ centered at 0.
 This property may also be obtained by noticing
 that we can also write $M(\alpha_1,\dots,\alpha_d)=\diag(\alpha_2,\dots,\alpha_d,\alpha_1)\,P_d$. From this, we easily deduce that 
\begin{align*}
\hbox{for all }k=1,\dots,d{-}1,\qquad \tr\big((e^{-i\theta}M+e^{i\theta}M^*)^k\big)=a_{k,0},
\end{align*}
and
\begin{align*}
\tr\big((e^{-i\theta}M+e^{i\theta}M^*)^d\big)=a_{d,0}+e^{-id\theta}\tr(M^d)+e^{id\theta}\tr(M^{*d}),
\end{align*}
for some values of $a_{k,0}$. Then the invariance follows from Theorem\,\ref{th2}.\medskip
 
We continue to assume that $\alpha_j>0$ for all $j$. We have seen that the boundary of $W(M)$ has straight parts iff the greatest root $w_*)$ of the polynomial $Q(w):=P(w)-P( w_*)$
 is also the greatest root of $P'$; therefore, a necessary condition for having straight parts is that the resultant $R(\alpha)=Res(Q,P')=0$. (In the following this resultant will be given up to a multiplication by some constant $\neq 0$.)\medskip
 
 \noindent$\bullet$ {\bf If $d=3$}
\begin{align*}
Q(w)&=w^3-\frac a4\,w+\frac b4,\quad \textrm{with }a=\alpha_1^2{+}\alpha_2^2{+}\alpha_3^2,\ b=\alpha_1\alpha_2\alpha_3,\\
R(\alpha)&=a^3-27b^2=(\alpha_1^2{+}\alpha_2^2{+}\alpha_3^2)^3-27\,\alpha_1^2\alpha_2^2\alpha_3^2.
\end{align*} It is easily seen that $R(\alpha)>0$ unless $\alpha_1=\alpha_2=\alpha_3$.
\smallskip

\noindent$\bullet$ {\bf If $d=4$}
\begin{align*}
Q(w)&=w^4-\frac a4\,w^2+\frac b{16},\quad \textrm{with }a=\alpha_1^2{+}\alpha_2^2{+}\alpha_3^2{+}\alpha_4^2,\ b=(\alpha_1\alpha_3{+}\alpha_2\alpha_4)^2,\\
R(\alpha)&=a^2-4\,b=(a{+}b)((\alpha_1{-}\alpha_3)^2{+}(\alpha_2{-}\alpha_4)^2)^3.
\end{align*} 
Then, the numerical range has straight parts on the boundary iff $\alpha_1=\alpha_3$ and $\alpha_2=\alpha_4$. In this case, we have
 \begin{align*}
T({-}\cos\theta ,{-}\sin\theta ,w)=\big(w^2{-}\frac{\alpha_1^{\,2}{+}\alpha_2^{\,2}}{4}{-}\frac{\alpha_1\alpha_2\cos(2\,\theta )}2\big)\big(w^2{-}\frac{\alpha_1^{\,2}{+}\alpha_2^{\,2}}{4}{-}\frac{\alpha_1\alpha_2\cos(2\,\theta{+\pi } )}2\big),
\end{align*}
and $M(\alpha_1,\dots,\alpha_4)$ is unitarily similar to the direct sum of the matrices $A_1=\begin{pmatrix} 0&\alpha_1\\\alpha_2 &0\end{pmatrix}$ and $A_2=i\,A_1$.\smallskip

\noindent$\bullet$ {\bf If $d=5$}
\begin{align*}
Q(w)&=w^5-\frac a4\,w^3+\frac b{16}\,w+\frac c{16},\quad \textrm{with }\\a&=\alpha_1^2{+}\cdots{+}\alpha_5^2,\ b=\alpha_1^2\alpha_3^2{+}\alpha_2^2\alpha_4^2{+}\alpha_3^2\alpha_5^2{+}\alpha_4^2\alpha_1^2{+}\alpha_5^2\alpha_2^2,\ c=\alpha_1\alpha_2\alpha_3\alpha_4\alpha_5,\\
R(\alpha)&=3125\,c^4-a(27\,a^4-225\,a^2b+500\,b^2)\,c^2+b^3(a^2{-}4b)^2.
\end{align*}
It can be seen by some cumbersome calculi that $R(\alpha)=0$ if $\alpha=(1,1,t,s,t)$ with $t>0$
and $s=1+(t-1)\frac{t+\varepsilon \sqrt{t^2+4}}2$, $\varepsilon =\pm1$. With $\varepsilon =-1$, we have noticed that the boundary of the numerical range has straight parts, but this is not the case if $\varepsilon =1$ and $t\neq 1$; then, $Q$ and $P'$ have a common root, but this is not the largest root of $P'$.

 \section{The dimension $d=4$ case.}
 
  We have seen in the previous section that, if $W(A)$ is invariant by rotation of angle $2\pi/d$ centered at $0$ with $d=3$, then $A$ is unitarily similar to a matrix $M(\alpha_1,\alpha_2,\alpha_3)$. This result does not generalize to $d\geq4$.
 Here we characterize (up to a unitary similarity) the matrices in $\C^{4,4}$ with numerical range invariant by rotation of angle $\frac{2\pi }4$ centered at 0 and not a disk. For that, it suffices to consider the case $\det(A)=-1$, 
 hence $A^4=I$ and $A$ can be written as
 \begin{align}\label{rot4}
 A=D+E,\quad \textrm{with  }D=\diag(1,i,-1,-i)\textrm{ and }E=\begin{pmatrix}
0&\alpha&\beta  &\gamma \\0 &0&\delta  &\varepsilon \\0 &0 &0 &\varphi \\0 &0 &0 &0
\end{pmatrix}
 \end{align}
with $\alpha\geq 0$, $\beta \geq 0$, $\gamma \geq 0$, $\delta ,\varepsilon ,\varphi \in\C$ ( if $\alpha\beta \gamma =0$ we can also assume $\varphi \leq 0$ or $i\varepsilon\leq 0$ ).
Indeed, the other matrices with this invariance property will be unitarily similar to  such $cA$ for some $c\neq 0$. We have seen that $A$ has this invariance if and only if\,: $\tr(A^*A^2)=0$ and
$\tr(A^*A^3)$=0. It is easily seen that
\begin{align*}
\tr(A^*A^2)&=\tr(E^*ED{+}EE^*D{+}E^*E^2),\\
\tr(A^*A^3)&=\tr(E^*D^2E{+}E^*DED{+}E^*ED^2{+}E^*DE^2{+}E^*EDE{+}E^*E^2D{+}E^*E^3)\\
&=\tr(EE^*D^2{+}E^*DED{+}E^*ED^2{+}E^2E^*D{+}EE^*ED{+}E^*E^2D{+}E^*E^3).
\end{align*}
Then, some tedious calculus gives
\begin{align*}
\tr(E^*ED)&=i\,\alpha^2-(\beta ^2+|\delta |^2)-i(\gamma ^2{+}|\varepsilon |^2{+}|\varphi |^2),\\
\tr(EE^*D)&=\alpha^2{+}\beta ^2{+}\gamma ^2+i(|\delta |^2{+}|\varepsilon |^2)-|\varphi |^2,\\
\tr(E^*E^2)&=\alpha\beta \gamma +\alpha\gamma \varepsilon +\beta \gamma \varphi +\delta \bar\varepsilon \varphi ,
\end{align*}
whence $\tr(A^*A^2)=\alpha^2{+}\gamma ^2{-}|\delta| ^2{-}|\varphi |^2{+}\alpha\beta \delta {+}\alpha\gamma \varepsilon {+}\beta \gamma \varphi {+}\delta \bar\varepsilon \varphi +i(\alpha^2
{+}|\delta |^2{-}\gamma ^2{-}|\varphi |^2)$.
\begin{align*}
\tr(EE^*D^2)&=\alpha^2{+}\beta ^2{+}\gamma ^2{-}|\delta |^2{-}|\varepsilon |^2{+|\varphi |^2}\\
\tr(E^*DED)&=i\,\alpha^2-\beta ^2-i\,|\delta |^2-i\,\gamma ^2+|\varepsilon|^2+i\,|\varphi |^2,\\
\tr(E^*ED^2)&=-\alpha^2+\beta ^2+|\delta |^2-\gamma ^2{-}|\varepsilon |^2{-}|\varphi |^2,\\
\tr(E^2E^*D)&=\alpha\beta \delta {+}\alpha\gamma \varepsilon {+}\beta \gamma \varphi {+}i\delta \bar\varepsilon \varphi ,\\
\tr(EE^*ED)&=i(\alpha\beta \delta {+}\alpha\gamma \varepsilon )-\beta \gamma \varphi -\delta \bar\varepsilon \varphi ,\\
\tr(E^*E^2D)&=-\alpha\beta \delta -i(\alpha\gamma \varepsilon {+}\beta \gamma \varphi {+}\delta \bar\varepsilon \varphi ),\\
\tr(E^*E^3)&=\alpha\gamma \delta \varphi ,
\end{align*}
whence $\tr(A^*A^3)=\beta ^2{-}|\varepsilon|^2 {+}\alpha\gamma \varepsilon {-}\delta \bar\varepsilon \varphi {+}\alpha\gamma \delta \varphi {+}i(\alpha^2{-}|\delta |^2{-}\gamma ^2{+}|\varphi |^2{+}\alpha\beta \delta {-}\beta \gamma \varphi )$.

Therefore, the numerical range of $A$ will be invariant by rotation of angle $\frac{2\pi }d$ centered at 0
if and only if
\begin{align*}
\alpha^2{+}\gamma ^2{-}|\delta| ^2{-}|\varphi |^2{+}\alpha\beta \delta {+}\alpha\gamma \varepsilon {+}\beta \gamma \varphi {+}\delta \bar\varepsilon \varphi +i(\alpha^2
{+}|\delta |^2{-}\gamma ^2{-}|\varphi |^2)&=0,\\
\textrm{and\hskip1.5cm  }\beta ^2{-}|\varepsilon|^2 {+}\alpha\gamma \varepsilon {-}\delta \bar\varepsilon \varphi {+}\alpha\gamma \delta \varphi {+}i(\alpha^2{-}|\delta |^2{-}\gamma ^2{+}|\varphi |^2{+}\alpha\beta \delta {-}\beta \gamma \varphi )&=0.
\end{align*}
We have not succeeded to describe all the solutions of these algebraic equations. We have $4$ equations in $\R$ with $9$ real unknowns, which suggests that the general solution depends on $5$ real parameters. Hereafter we consider some subfamilies:\bigskip

\noindent$\bullet$ With $\alpha=\varphi =0$, we can choose $\varepsilon i \leq 0$, then we get the family:

 with a parameter $b\geq 0$\,: $\alpha=0$, $\beta =b$, $\gamma=0$, $\delta=0$, $\varepsilon =i\,b$, $\varphi =0$. Here the matrix $A$ is (unitarily similar to) the direct sum of the matrices $A_1$ and $iA_1$ with
  $A_1=\begin{pmatrix}1 &b\\0&-1\end{pmatrix}$.\medskip

\noindent$\bullet$ With $\varepsilon =\beta =0$, we can choose $\varphi\leq 0$, then we get two families of solutions of these equations\,:

with a parameter $a\geq 0$: $\alpha=a$, $\beta =0$, $\gamma=\sqrt{\frac{2\,a^2}{2{+}a^2}}$, $\delta=i\gamma $, $\varepsilon =0$, $\varphi =-a$,

with a parameter $a\in[ 0,\sqrt2[$: $\alpha=a$, $\beta =0$, $\gamma=\sqrt{\frac{2\,a^2}{2{-}a^2}}$, $\delta=-i\gamma $, $\varepsilon =0$, $\varphi =-a$.

\begin{align*}
A=\begin{pmatrix}
1&a&0 &c\\0 &i &ic &0\\0 &0 &-1 &-a\\0 &0 &0 &-i
\end{pmatrix},\quad \textrm{with \quad}c=\sqrt{2a^2/(2\pm a^2)},
\end{align*}

These families are distinct from the family of $M(\alpha_1,\dots,\alpha_4)$ studied in the previous section. For instance, for $a=1$ and $c=\sqrt{2/3}$,
the (numerically computed) singular values of $A$ are $2.077\dots$, $1.260\dots$, $1.149\dots$, $0.332\dots$,
and $\|A^2\|\simeq2.4142\dots$ If $A$ were unitarily similar to a matrix $M(\alpha_1,\dots,\alpha_4)$, then $\|A^2\|$ would be the product of two of these singular values; clearly, this is not the case.

\noindent$\bullet$ With two parameters $\alpha>0$ and $\gamma >0$, we get also two families of solutions of these equations\,: $\beta =\frac{\gamma ^2-\alpha^2}{\alpha\gamma}$, $\delta =\gamma $,
$\varepsilon =\alpha\gamma \pm\beta $, $\varphi =-\alpha$.
\medskip

\noindent$\bullet$ With two parameters $\rho >0$, $\theta \in]-\frac\pi 4,\frac\pi 4[$, we have the family
$\alpha=\gamma =0$, $\delta =\rho \,e^{i\theta }$, $\varepsilon =\frac2{\sqrt{\cos2\theta }}$, $\varphi =\rho \sqrt{\tan(\theta{+}\frac\pi 4)}$, $\beta =\sqrt{\rho ^2{+}\varepsilon ^2{+}\varphi ^2}$.

\noindent$\bullet$ There is also the family of matrices $M(\alpha_1,\dots,\alpha_4)$ considered in the previous section. They are unitarily similar to matrices of the form \eqref{rot4} and the entries $\alpha,\dots,\varphi $ could be expressed as functions of $\alpha_1,\dots,\alpha_4$, but the corresponding  formulae would be very involved.

\medskip

Thus, for $d\geq 4$ there are many other classes of matrices with numerical range invariant by these rotations. 
 
\section{Comments}

After the deposit of a first version of this paper on arXiv, from the reading of the article \cite{gifW} on a similar but more general subject, I have discovered my guilty ignorance of interesting previous works. To my knowledge, Theorem\,1 is new, it is limited to the case where we have the same $d$ for the dimension of the matrix and for the invariance under rotation of angle $2\pi /d$. Theorem\,2,  which is its translation in terms of traces,  was obtained for the $3\times3$ case in \cite{har}, and for the $4\times4$ case in \cite{dea}, these last authors have also the corresponding sufficient condition in the general $d\times d$ case, but not the necessary condition, this condition could be deduced from \cite[Theorem 4.2]{gifW}. I have not found in the literature the equivalent of Section\,4. The invariance under rotation of the matrices $M(\alpha_1,\dots,\alpha_d)$ is known since the Ph.D. dissertation of Issos, and Lemma\,\ref{lem6} may be found in \cite{har}. A different characterization of weight shift matrices with line segments on the boundary is given in \cite{tsaiwu}.
\bigskip

\noindent{\bf Acknowledgement. }{\em The author would like to express his thanks to Professor Jean-Claude Raoult for interesting discussions concerning the proof of Theorem 1, and to Professor Hugo Woederman for pointing out the article  \cite{gifW}.
}

\end{document}